\documentclass[12pt]{amsart}
\usepackage{amssymb}
\newtheorem{theorem}{Theorem}
\newtheorem{lemma}[theorem]{Lemma}

\newtheorem{corollary}[theorem]{Corollary}
\newtheorem{example}[theorem]{Example}
\newtheorem{question}{Question}

\newcommand{\dom}{\text{dom}}

\title{$G_\delta$ covers of compact spaces}
\author{Santi Spadaro}

\address{
Department of Mathematics and Computer Science \\
University of Catania \\
Citt\'a universitaria\\  
viale A. Doria 6 \\
95125 Catania, Italy}
\email{santidspadaro@gmail.com}

\author{Paul Szeptycki}

\address{Department of Mathematics\\
Faculty of Science and Engineering\\
York University\\
Toronto, ON,  M3J 1P3 Canada}
\email{szeptyck@yorku.ca}



\begin{document}

\maketitle

\begin{abstract} 
We solve a long standing question due to Arhangel'skii by constructing a compact space which has a $G_\delta$ cover with no continuum-sized ($G_\delta$)-dense subcollection. We also prove that in a countably compact weakly Lindel\"of normal space of countable tightness, every $G_\delta$ cover has a $\mathfrak{c}$-sized subcollection with a $G_\delta$-dense union and that in a Lindel\"of space with a base of multiplicity continuum, every $G_\delta$ cover has a continuum sized subcover. We finally apply our results to obtain a bound on the cardinality of homogeneous spaces which refines De La Vega's celebrated theorem on the cardinality of homogeneous compacta of countable tightness.
\end{abstract}

\section{Introduction}  Alexandroff and Urysohn asked, in 1923, if the continuum is a bound on the cardinality of compact Hausdorff first-countable spaces. The celebrated solution by Arhangel'skii \cite{A1} established that the cardinality of any Hausdorff space is bounded by a function of the Lindel\"of degree and character, namely we have the inequality $|X|\leq 2^{\chi(X)L(X)}$. This result was improved in many directions and some still outstanding open problems guide ongoing research. Much of the work in this area is concerned with establishing similar bounds on the cardinality of $X$ from more general cardinal invariants obtained by weakening the Lindel\"of degree and character in conjunction with perhaps strengthening the separation axioms. For example, the bound of Arhangel'skii-Sapirovskii that for $T_2$ spaces $|X|\leq 2^{\psi(X)t(X)L(X)}$ and the Bell-Ginsburg-Woods inequality for normal spaces that $|X|\leq 2^{\chi(X)wL(X)}$ (see \cite{BGW}) are in this spirit. Recall that a space is said to be weakly Lindel\"of if for each open cover ${\mathcal U}$ there is a countable ${\mathcal V}\subseteq {\mathcal U}$ such that $\bigcup{\mathcal V}$ is dense and $wL(X)$ (the weak Lindel\"of number of $X$) is defined as the minimum cardinal $\kappa$ such that every open cover has a subcollection of cardinality $\leq \kappa$ with dense union. See \cite{H} for a more detailed survey on Arhangel'skii's solution and subsequent research. 

Two questions attributed to Arhangels'kii (see \cite{J} and \cite{FW} for published references to these questions)  go in a completely different direction, asking about cardinal invariants for the $G_\delta$ topology. Recall that given a space $X$ we denote by $X_\delta$ the topology with the same underlying set $X$, generated by the $G_\delta$ subsets of $X$

\begin{question}\label{QA} Let $X$ be compact $T_2$.
\begin{enumerate} 
\item Is $L(X_\delta)\leq 2^{\aleph_0}$?
\item Is $wL(X_\delta)\leq 2^{\aleph_0}$?
\end{enumerate}
\end{question}

A positive solution to either question would also answer the Alexandroff-Urysohn question. Indeed, for a first countable compactum $X$, the space $X_\delta$ would be discrete, and in this case both the Lindel\"of and weak Lindel\"of degree coincide with the cardinality of $X$. 

A negative answer to the first question has been known for some time. For example, Mycielski proved that if $\kappa$ is less than the first inaccessible, then $e(\omega^\kappa)=\kappa$ \cite{M}. Recall that $e(X)$ denotes the {\em extent of $X$}, that is the supremum of cardinalities of closed discrete subsets of $X$. Hence $L(\omega^\kappa)=\kappa$ if $\kappa$ is less than the first inaccessible. Therefore, if one considers $X=(2^\omega)^\kappa$, then it follows that $L(X_\delta)=\kappa$ as well. Gorelic has similar results for a larger class of cardinals $\kappa$, including that  $e(\omega^{2^\kappa})=\kappa$ if $\kappa$ is less than the first measurable cardinal \cite{G}. As a result, we have that the Lindel\"of degree of compacta under the $G_\delta$ topology can be arbitrarily large below the first measurable. 

\begin{question} Is the first measurable a bound on the Lindel\"of degree of compacta under the $G_\delta$ topology?
\end{question}

However, Question $\ref{QA}$, (2) has remained open until now.

There has been a fair amount of work in a positive direction on Arhangel'skii's questions. For example, Juh\'asz proved in \cite{J} that $wL(X_\delta)\leq 2^{\aleph_0}$ for every compact ccc space $X$, using the Erd\"os-Rado theorem. The first-named author generalized this in \cite{S} to prove that $wL(X_\delta) \leq 2^{\aleph_0}$ for all spaces $X$ such that player II has a winning strategy in $G^{\omega_1}(\mathcal{O}, \mathcal{O}_D)$, that is the two-player game in $\omega_1$ many innings where at inning $\alpha<\omega_1$, player one chooses a maximal family of non-empty pairwise disjoint open sets $\mathcal{U}_\alpha$ and player two chooses $U_\alpha \in \mathcal{U}_\alpha$ and player two wins if $\bigcup \{U_\alpha: \alpha < \omega_1 \}$ is dense in $X$. In \cite{FW}, Fleischmann and Williams proved that $L(X_\delta) \leq 2^{\aleph_0}$ for every compact linearly ordered space $X$ and Pytkeev proved in \cite{P} that $L(X_\delta)\leq 2^{\aleph_0}$ for every compact countably tight space $X$. Carlson, Porter and Ridderbos generalized the latter result in \cite{CPR} by proving that $L(X_\delta) \leq 2^{F(X) t(X) L(X)}$, where $F(X)$ is the supremum of the cardinalities of the free sequences in $X$. This is actually an improvement only in the non-compact realm, since for compact spaces $F(X)=t(X)$.

In Section 2 we answer Arhangel'skii's question by constructing a compact subspace of $(2^\omega)^{{\mathfrak c}^+}$ and a $G_\delta$ cover with no ${\mathfrak c}$ sized subcollection with dense union in the $G_\delta$ topology.

In Section 3 we provide a few more positive result about covering properties of the $G_\delta$ topology. In particular we prove that in a countably compact weakly Lindel\"of normal space of countable tightness, every $G_\delta$ cover has a $\mathfrak{c}$-sized subcollection with a $G_\delta$-dense union and that in a Lindel\"of space with a base of multiplicity continuum, every $G_\delta$ cover has a continuum sized subcover. 

In Section 4 we apply one of the results from section 3 to extend De la Vega's theorem on the cardinality of homogeneous compacta to the realm of countably compact spaces.

In our proofs we will often use elementary submodels of the structure $(H(\mu), \epsilon)$. Dow's survey \cite{D} is enough to read our paper, and we give a brief informal refresher here. Recall that $H(\mu)$ is the set of all sets whose transitive closure has cardinality smaller than $\mu$. When $\mu$ is regular uncountable, $H(\mu)$ is known to satisfy all axioms of set theory, except the power set axiom. We say, informally, that a formula is satisfied by a set $S$ if it is true when all existential quantifiers are restricted to $S$. A set $M \subset H(\mu)$ is said to be an elementary submodel of $H(\mu)$ (and we write $M \prec H(\mu)$) if a formula with parameters in $M$ is satisfied by $H(\mu)$ if and only if it is satisfied by $M$. 

The downward L\"owenheim-Skolem theorem guarantees that for every $S \subset H(\mu)$, there is an elementary submodel $M \prec H(\mu)$ such that $|M| \leq |S| \cdot \omega$ and $S \subset M$. This theorem is sufficient for many applications, but it is often useful (especially in cardinal bounds for topological spaces) to have the following closure property. We say that $M$ is \emph{$\kappa$-closed} if for every $S \subset M$ such that $|S| \leq \kappa$ we have $S \in M$. For large enough regular $\mu$ and for every countable set $S \subset H(\mu)$ there is always a $\kappa$-closed elementary submodel $M \prec H(\mu)$ such that $|M|=2^{\kappa}$ and $S \subset M$.

The following theorem is also used often: let $M \prec H(\mu)$ such that $\kappa + 1 \subset M$ and $S \in M$ be such that $|S| \leq \kappa$. Then $S \subset M$.

All spaces are assumed to be Hausdorff. Undefined notions can be found in \cite{Ku} for set theory and \cite{E} for topology. However, our notation regarding cardinal functions follows Juh\'asz's book \cite{Ju}.

\section{Counterexamples to Arhangel'skii's questions}

We first give a direct alternate proof of the known result that the Lindel\"of degree of 
$(2^\omega)^{{\mathfrak c}^+}$ under the $G_\delta$ topology is ${\mathfrak c}^+$. 
%note that if there is an example of weight ${\mathfrak c}^+$ then the Cantor cube and the Tychonoff cube of weight ${\mathfrak c}^+$ are examples. 

%\begin{theorem} The following are equivalent
%\begin{enumerate}
%\item There is a compactum $K$ of weight ${\mathfrak c}^+$ whose Lindel\"of degree with respect to the $G_\delta$-topology is $>{\mathfrak c}$
%\item $[0,1]^{{\mathfrak c}^+}$ with the $G_\delta$-topology has Lindel\"of degree $>{\mathfrak c}$
%\item $(2^\omega)^{{\mathfrak c}^+}$ with the $G_\delta$-topology has Lindel\"of degree $>{\mathfrak c}$
%\item $2^{{\mathfrak c}^+}$ with the $G_\delta$-topology has Lindel\"of degree $>{\mathfrak c}$.
%\end{enumerate}
%\end{theorem}

%\begin{proof} 
%$(1)\Rightarrow (2)$: Any compactum of weight ${\mathfrak c}^+$ embeds as a closed subset of $[0,1]^{\mathfrak c^+}$. And small Lindel\"of degree of $[0,1]^{\mathfrak c^+}$ in the $G_\delta$ topology is closed hereditary.  

%$(2)\Rightarrow (3)$ $[0,1]$ in the $G_{\delta}$ topology is discrete, hence homeomorphic to $2^\omega$ in the $G_\delta$ topology. So $[0,1]^{\mathfrak c^+}$ and $(2^\omega)^{\mathfrak c^+}$ are homeomorphic wrt the $G_\delta$ topology. 

%$(3)\Rightarrow (4)$. Also $(2^\omega)^{\mathfrak c^+}$ is homeomorphic to $2^{\mathfrak c^+}$ with respect to the $G_\delta$ topology.

%And clearly (4) implies (1). 
%\end{proof}

\begin{theorem} \label{mycielthm}
The Lindel\"of degree of $(2^\omega)^{{\mathfrak c}^+}$ under the $G_\delta$ topology is ${\mathfrak c}^+$.
\end{theorem}

\begin{proof} Note first that $(2^{\omega})^{\mathfrak{c}^+}$ with the $G_\delta$ topology is homeomorphic to $(D(\mathfrak{c})^{\mathfrak{c}^+})_\delta$, where $D(\mathfrak{c})$ is a discrete set of size continuum. We need to exhibit a cover with no subcover of size ${\mathfrak c}$.  For any countable partial function $s:{\mathfrak c}^+\rightarrow D(\mathfrak{c})$, the set $$[s]=\{f\in (D(\mathfrak{c}))^{{\mathfrak c}^+} : s\subseteq f\}$$ is open in the $G_\delta$ topology.
Let $${\mathcal U}=\{[s]: \dom(s)\in[\mathfrak{c}^+]^{\aleph_0}\text{ and }s:\dom(s)\rightarrow D(\mathfrak{c}) \text{ is not 1-1}\}$$

Note that since any function $f:{\mathfrak c}^+ \to D(\mathfrak{c})$ fails to be 1-1 on some countable subset, $((D(\mathfrak{c})^{{\mathfrak c}^+})_\delta$ is covered by ${\mathcal U}$. But if ${\mathcal V}\subseteq {\mathcal U}$ has cardinality $\leq {\mathfrak c}$, then there is an $\alpha<{\mathfrak c}^+$ such that $\dom(s)\subseteq \alpha$ for all $[s]\in{\mathcal V}$. Fix $g\in D(\mathfrak{c})^{{\mathfrak c}^+}$ such that $g\upharpoonright \alpha$ is 1-1, then $g$ is not covered by ${\mathcal V}$. Thus the Lindel\"of degree of $(2^\omega)^{{\mathfrak c}^+}$ under the $G_\delta$ topology is ${\mathfrak c}^+$. 

\end{proof}

The cover ${\mathcal U}$ could also have been chosen slightly differently. E.g., one could have also considered those $[s]$ with countable support where $s$ is not {\em finite-to-one}, or those $[s]$ with {\em finite} support and $s$ not 1-1 and the proof would also work. Indeed, the latter option gives us an open cover of $D(\mathfrak{c})^{{\mathfrak c}^+}$ with its usual product topology. Therefore, as a corollary to the proof we obtain Mycielski's result that $D(\kappa)^{\kappa^+}$ contains a closed discrete subset of size $\kappa^+$ \cite{M}.

\begin{corollary}{\em(Mycielski)} For any cardinal $\kappa$, the product topology on $D(\kappa)^{\kappa^+}$ has Lindel\"of degree $\kappa^+$ 
\end{corollary}

We now construct a compact space $X$ such that $wL(X_\delta) > 2^\omega$ in ZFC, which solves Arhangel'skii's question.

\begin{theorem} \label{mainex}
Suppose there is a compact space which has a partition into $\kappa$ many $G_\delta$ sets. Then there is a compact space $X$ admitting a $G_\delta$-cover of $X$ with no $\kappa$-sized dense subcollection (in particular, $wL(X_\delta) >\kappa$).
\end{theorem}

\begin{proof}
Let $\tilde{K}$ be a compact space having a partition into $\kappa$ many $G_\delta$ sets, set $K=\tilde{K} \times 2$ and let $\{G_\alpha: \alpha < \kappa\}$ be a partition of $K$ into $\kappa$ many $G_\delta$ sets. Moreover, let $\{V_0, V_1 \}$ be a partition of $K$ into a pair of non-empty clopen sets. Without loss we can assume that $\{G_\alpha \cap V_1: \alpha < \kappa\}$ is a pairwise disjoint family of $G_\delta$ sets of cardinality $\kappa$. We define $X \subset (K)^{\mathfrak{\kappa}^+}$ as an inverse limit of compacta $X_\alpha \subset (K)^\alpha$. 

The $X_\alpha$'s are defined recursively preserving the following two conditions for every $\alpha < \kappa^+$

\begin{enumerate}
\item $X_\alpha$ is a closed subset of $K^\alpha$.
\item $(V_1)^\alpha \subset X_\alpha$.
\end{enumerate}

The base case is $X_{\kappa}=(K)^{\kappa}$. For $\alpha > \kappa$ limit, let $X_\alpha$ be the inverse limit of the previously defined $X_\beta$'s. Note that the inductive hypotheses are satisfied.

Suppose now $\alpha=\beta+1$ and $X_\beta$ has already been defined and use $(V_1)^\beta \subset X_\beta$ to choose $f_\beta \in X_\beta$ such that $f_\beta(\gamma)$ and $f_\beta(\delta)$ don't belong to the same $G_\tau$, whenever $\gamma < \delta < \alpha$. 

Now let $X_{\beta+1}=\{g: \beta+1 \to K:  g \upharpoonright \beta \in X_\beta \wedge (g(\beta) \in V_0 \Rightarrow g \upharpoonright \beta=f_\beta) \}$.

Since $X_{\beta+1}=(X_\beta \times V_1) \cup (\{f_\beta\} \times V_0)$, the two inductive hypotheses are preserved.

Finally let $X=X_{\kappa^+}$. 

Given a partial function $s: dom(s) \to K$, where $dom(s) \in [\kappa^+]^\omega$, let $<s>=\prod \{W_\alpha: \alpha \in \kappa^+\}$, where $W_\alpha=G_{s(\alpha)}$ if $\alpha \in dom(s)$ and $W_\alpha = K$ otherwise. The set $\mathcal{U}=\{<s>: s$ is not one-to-one $\}$ is a $G_\delta$ cover of $X$ such that no $\kappa$-sized subcollection has a dense union. Indeed, let $\mathcal{V} \subset \mathcal{U}$ be a $\kappa$-sized subcollection And let $\alpha <\kappa^+$ be an ordinal such that $dom(s) \subset \alpha$ for every $<s> \in \mathcal{V}$. Thus $f_\alpha \notin \bigcup \mathcal{V}$. Now consider the basic open set $W:=\{g \in X: g(\alpha) \in V_0 \}$. Then $W \subset \{g \in X: g \upharpoonright \alpha = f_\alpha \}$, hence $W \cap (\bigcup \mathcal{V})=\emptyset$, as we wanted.

\end{proof}

\begin{corollary} \label{maincor}
There is a compact space $X$ such that $wL(X_\delta)=\mathfrak{c}^+$.
\end{corollary}

\begin{proof}
Simply set $K=2^\omega$ in the construction of Theorem $\ref{mainex}$ and note that every point of $K$ is a $G_\delta$ set.
\end{proof}

Theorem $\ref{mainex}$ suggests a way of getting a compact space whose $G_\delta$ topology has weak Lindel\"of number greater than the successor of the continuum, provided that there exists a compact space having a partition into $G_\delta$ sets of cardinality $\mathfrak{c}^+$. Thus we may ask:

\begin{question}
Is there a compact space having a partition into $\mathfrak{c}^+$ many $G_\delta$ sets?
\end{question}

While Arhangel'skii \cite{A2} showed that no compactum can be partitioned into more that $2^{\aleph_0}$ closed $G_\delta$'s, we do not know if there is a bound on the size of partitions of compacta into $G_\delta$ subsets. This suggests the following question:

\begin{question} \label{partquest}
Is there in ZFC a cardinal $\kappa$ so that for any compact space $X$ and any partition $P$ of $X$ into $G_\delta$ subsets, we have that $|P|<\kappa$?
\end{question}

Even if Question $\ref{partquest}$ had a positive answer, this would not exclude the possibility of the existence of compact spaces with arbitrarily large weak Lindel\"of number in their $G_\delta$ topologies. So we finish with the following more general question:

\begin{question} \label{questbound}
Is there any bound on the weak Lindel\"of number of the $G_\delta$ topology on a compact space?
\end{question}

We are also intrigued about the possibility of restricting Arhangel'skii's problem to compact spaces with some additional structure. Every compact group has the countable chain condition, so using Juh\'asz's result from \cite{J} we get that $wL(X_\delta) \leq c(X_\delta) \leq 2^{\aleph_0}$ for every compact group $X$. Recall that space is \emph{homogeneous} if for every pair of points $x, y \in X$ there is a homeomorphism $f: X \to X$ such that $f(x)=y$. Every topological group is a homogeneous space.

\begin{question}
Is there a compact homogenous space $X$ such that $wL(X_\delta) > 2^{\aleph_0}$?
\end{question}

Actually, we don't even know whether the example from Corollary $\ref{maincor}$ can be made homogenous. If it could, it would provide an answer to van Douwen's long standing question about the existence of a compact homogenous space of cellularity larger than the continuum (see \cite{K}). As a matter of fact, the cellularity of our example is ${\mathfrak c}^+$. Indeed, the clopen sets $W_\alpha=\pi_{\{\alpha\}}^{-1}(U_0)=\{f\in X:f(\alpha)\in U_0\}$ are pairwise disjoint. To see this, suppose that $\beta<\alpha$ and recall that $f_\alpha$ was chosen to be a 1-1 function in $(U_1)^\alpha$. And by the construction, if $f\in X$ and $f(\alpha)=0$ then 
$f\upharpoonright \alpha=f_\alpha$. And so for any $\beta<\alpha$ and any $f\in W_\alpha$ we have that $f(\beta)=1$. I.e., $f\not\in W_\beta$ and so $W_\alpha\cap W_\beta=\emptyset$.

\section{Bounds for the $G_\kappa$ modification}

Given a space $X$ we denote by $X_\kappa$ the topology on $X$ generated by the $G_\kappa$-subsets of $X$ (that is, the intersections of $\kappa$-sized families of open subsets of $X$). It is natural to ask what properties are preserved when passing from $X$ to $X_\kappa$ and whether cardinal invariants of $X_\kappa$ can be bound in terms of cardinal invariants of $X$. There has been a fair amount of work in the past on this general question, especially for chain conditions and covering properties (see for example \cite{LR}, \cite{KMS}, \cite{J}, \cite{FW}, \cite{Ge}, \cite{S}). The aim of this section is to present some preservation results that are related to Arhangel'skii's problems mentioned in the introduction.

Recall that $wL_c(X)$ is defined as the minimum cardinal $\kappa$ such that for every closed set $F \subset X$ and for every family $\mathcal{U}$ of open sets of $X$ covering $F$ there is a $\kappa$-sized subfamily $\mathcal{V}$ of $\mathcal{U}$ such that $F \subset \overline{\bigcup \mathcal{V}}$. 

It's well known and easy to prove that $wL_c(X)=wL(X)$ for every normal space.

A set $G \subset X$ is called a $G^c_\kappa$ set if there is a family $\{U_\alpha: \alpha < \kappa \}$ of open subsets of $X$ such that $G=\bigcap \{U_\alpha: \alpha < \kappa\}=\bigcap \{\overline{U_\alpha}: \alpha < \kappa \}$.

Given a space $X$, we denote with $X^c_\kappa$ the topology generated by the $G^c_\kappa$ subsets of $X$. Clearly, if $X$ is regular then $X^c_\kappa=X_\kappa$.

\begin{theorem} \label{wLtheorem}
Let $X$ be an initially $\kappa$-compact space such that $t(X) wL_c(X) \leq \kappa$. Then $wL(X^c_\kappa) \leq 2^\kappa$.
\end{theorem}

%Proof can be simplified. No need for P and O, one open set is enough.

\begin{proof}

Let $\mathcal{F}$ be a cover of $X$ by $G^c_\kappa$ sets. Let $M$ be a $\kappa$-closed elementary submodel of $H(\theta)$ such that $X, \mathcal{F} \in M$, $\kappa+1 \subset M$ and $|M|  \leq 2^\kappa$.

\noindent {\bf Claim.} $\mathcal{F} \cap M$ covers $\overline{X \cap M}$.

\begin{proof}[Proof of Claim]
Fix $x \in \overline{X \cap M}$ and let $F$ be an element of $\mathcal{F}$ containing $x$. Let $C$ be a subset of $X \cap M$ of cardinality $\kappa$ such that $x \in \overline{C}$. Let $\{U_\alpha: \alpha < \kappa \}$ be a $\kappa$-sized sequence of open sets such that $F=\bigcap \{U_\alpha: \alpha < \kappa \}=\bigcap \{\overline{U_\alpha}: \alpha < \kappa\}$. Let $C_\alpha=U_\alpha \cap C$. Since $U_\alpha$ is a neighbourhood of $x$ we have $x \in \overline{C_\alpha}$. Since $M$ is $\kappa$-closed we have $C_\alpha \in M$. Moreover, $\bigcap_{\alpha<\kappa} \overline{C_\alpha} \in M$. Let $B=\bigcap_{\alpha < \kappa} \overline{C_\alpha}$ and note that $H(\theta) \models (\exists G \in \mathcal{F})(B \subset G)$. Since $B \in M$, it follows by elementarity that $M \models (\exists G \in \mathcal{F})(B \subset G)$. Hence there is $G \in \mathcal{F} \cap M$ such that $B \subset G$, and since $x \in B \subset G$ we get what we wanted.
\renewcommand{\qedsymbol}{$\triangle$}
\end{proof}

Let us now prove that $\bigcup (\mathcal{F} \cap M)$ is $G^c_\kappa$-dense in $X$. 

If this were not the case, there would be a $G^c_\kappa$-subset $G \subset X$ such that $G \cap \bigcup (\mathcal{F} \cap M)=\emptyset$.

Let $\{V_\alpha: \alpha < \kappa \}$ be a sequence of open subsets of $X$ such that $G=\bigcap \{V_\alpha: \alpha < \kappa\}=\bigcap \{\overline{V_\alpha}: \alpha < \kappa \}$. Using initial $\kappa$ compactness and the above claim we can find, for every $x \in \overline{X \cap M}$ an open neighbourhood $U_x \in M$ of the point $x$ and a finite subset $F_x \subset \kappa$ such that $U_x \cap \bigcap \{V_\alpha: \alpha \in F_x\}=\emptyset$. For every $F \in [\kappa]^{<\omega}$ let $U_F=\bigcup \{U_x: F_x=F\}$. Then $\{U_F: F \in [\kappa]^{<\omega}\}$ is a $\kappa$-sized open cover of the initially $\kappa$-compact space $\overline{X \cap M}$. Hence we can find a finite subset $\mathcal{H} \subset [\kappa]^{<\omega}$ such that $\{U_F: F \in \mathcal{H} \}$ covers $\overline{X \cap M}$. But then $\mathcal{U}=\{U_x: F_x \in \mathcal{H}\}$ is an open cover of $\overline{X \cap M}$. Using the fact that $wL_c(X) \leq \kappa$ we can find a subcollection $\mathcal{V} \in [\mathcal{U}]^\kappa$ such that $X \cap M \subset \overline{X \cap M} \subset \overline{\bigcup \mathcal{V}}$. Note that $\mathcal{V} \subset \mathcal{U} \subset M$ and $M$ is $\kappa$-closed, so $\mathcal{V} \in M$ and hence $M \models X \subset \overline{\bigcup \mathcal{V}}$. By elementarity it follows that $H(\theta) \models X \subset \overline{\bigcup \mathcal{V}}$, but that contradicts the fact that the open set $\bigcap \{\bigcap \{V_\alpha: \alpha \in F \}: F \in \mathcal{H}\}$ misses every element of $\mathcal{U}$.
\end{proof}

\begin{corollary}
Let $X$ be an initially $\kappa$-compact regular space such that $t(X) wL_c(X) \leq \kappa$. Then $wL(X_\kappa) \leq 2^\kappa$.
\end{corollary}

\begin{corollary}
Let $X$ be a normal initally $\kappa$-compact space such that $wL(X) t(X) \leq \kappa$. Then $wL(X_\kappa) \leq 2^\kappa$.
\end{corollary}

Note that all assumptions are essential in Theorem $\ref{wLtheorem}$:

\begin{enumerate}

\item Let $\kappa$ be a cardinal of uncountable cofinality. To find countably compact spaces of countable tightness $X$ such that $wL(X_\delta)$ can be arbitrarily large, let $X=\{\alpha < \kappa: cf(\alpha)=\aleph_0 \}$, with the topology induced by the order topology on $\kappa$.

\item To find regular spaces $X$ such that $wL_c(X) t(X)=\aleph_0$ and yet $wL(X_\delta)$ can be arbitrarily large, let $X$ be Uspenskij's example of a $\sigma$-closed discrete dense subset of a $\sigma$-product of $\kappa$ many copies of the unit interval from \cite{U}. Being dense in $\mathbb{I}^\kappa$, the space $X$ has the countable chain condition and hence $wL_c(X)=\aleph_0$. The tightness of $X$ is countable because a $\sigma$-product of intervals is even Fr\'echet-Urysohn and $\sigma$-closed discrete implies points $G_\delta$, hence $X_\delta$ is discrete. Since a $\sigma$-product of intervals has density and cardinality $\kappa$ we actually have $wL(X_\delta)=\kappa$.

\item An example of a compact Hausdorff (and hence countably compact normal) space such that $wL_c(X)=\omega$ and $wL(X_\delta) > \mathfrak{c}$ is provided by Corollary $\ref{maincor}$.
\end{enumerate}

\begin{question}
Is there a countably compact normal weakly Lindel\"of space of countable tightness such that $L(X_\delta) > 2^{\aleph_0}$.
\end{question}

Recall that the \emph{multiplicity} of a base $\mathcal{B}$ is the minimum cardinal $\kappa$ such that for every point $x \in X$, the set $\{B \in \mathcal{B}: x \in B \}$ has cardinality at most $\kappa$.

\begin{theorem}
Let $X$ be a space such that $L(X)=\kappa$ and $X$ has a base of multiplicity $2^\kappa$. Then $L(X_\kappa) \leq 2^{\kappa}$.
\end{theorem}

\begin{proof}
Fix a base $\mathcal{B}$ for $X$ having multiplicity $\kappa$ and let $\mathcal{U}$ be a cover of $X$ by $G_\kappa$ sets.

Let $\theta$ be a large enough regular cardinal and let $M \prec H(\theta)$ be a $\kappa$-closed elementary submodel of cardinality $2^\kappa$ such that $X, \mathcal{B}, \mathcal{U} \in M$ and $2^\kappa+1 \subset M$.

{\bf Claim.} $\mathcal{U} \cap M$ covers $\overline{X \cap M}$.

\begin{proof}[Proof of Claim]
Fix $x \in \overline{X \cap M}$ and let $U \in \mathcal{U}$ be such that $x \in U$. Let $\{U_\alpha: \alpha < \kappa \}$ be a sequence of open sets such that $U=\bigcap \{U_\alpha: \alpha < \kappa \}$.  For every $\alpha < \kappa$ Let $B_\alpha \in \mathcal{B}$ be such that $x \in B_\alpha \subset U_\alpha$. Fix $x_\alpha \in B_\alpha \cap M$ and note that $\{B \in \mathcal{B}: x_\alpha \in B\}$ is an element of $M$ and has size $2^\kappa$. Hence $\{B \in \mathcal{B}: x_\alpha \in B\} \subset M$. It follows that $B_\alpha \in M$ for every $\alpha < \kappa$ and thus $\bigcap_{\alpha<\kappa} B_\alpha \in M$, as $M$ is $\kappa$-closed. Set $B:=\bigcap_{\alpha < \kappa} B_\alpha$. Note that $H(\theta) \models (\exists U)(U \in \mathcal{U} \wedge B \subset U)$. Since $B \in M$, by elementarity we have $M \models (\exists U)(U \in \mathcal{U} \wedge B \subset U)$, or equivalently, there is $U \in \mathcal{U} \cap M$ such that $B \subset U$. Since $x \in B \subset U$ we get what we wanted.
\renewcommand{\qedsymbol}{$\triangle$}
\end{proof}

Let us now prove that $\mathcal{U} \cap M$ actually covers $X$.

Suppose this were not true and let $p \in X \setminus \bigcup (\mathcal{U} \cap M)$. For every $x \in \overline{X \cap M}$, let $U_x \in \mathcal{U} \cap M$ be such that $x \in U_x$. For every $x \in \overline{X \cap M}$, we can find a sequence of open sets $\{U^x_\alpha: \alpha < \kappa \} \in M$ such that $U_x=\bigcap \{U^x_\alpha: \alpha <\kappa \}$. Since $\kappa +1 \subset M$ we actually have $\{U^x_\alpha: \alpha < \kappa \} \subset M$. For every $x \in X \cap M$, there is $\alpha_x < \kappa$ such that $p \notin U^x_{\alpha_x}$. Finally $\mathcal{V}:=\{U^x_{\alpha_x}: x \in \overline{X \cap M} \}$ is an open cover of the subspace $\overline{X \cap M}$, which has Lindel\"of number $\kappa$, and hence there is a $\kappa$-sized $\mathcal{C} \subset \mathcal{V}$ such that $X \cap M \subset \overline{X \cap M} \subset \bigcup \mathcal{C}$. But, since $\mathcal{C} \in M$ by $\kappa$-closedness, this implies that $M \models X \subset \bigcup \mathcal{C}$ and hence $H(\theta) \models X \subset \bigcup \mathcal{C}$, which is a contradiction because $p \notin \bigcup \mathcal{C}$.

\end{proof}

\begin{corollary}
Let $X$ be a Lindel\"of space with a point-countable base. Then $L(X_\delta) \leq 2^{\aleph_0}$.
\end{corollary}

\section{Applications to homogeneous spaces}

De la Vega's theorem \cite{De} states that the cardinality of every compact homogeneous space of countable tightness is at most the continuum. 

We will use the results from Section 3 to extend De La Vega's theorem to the realm of countably compact spaces. It's not enough to replace compact with countably compact. Indeed, let $\kappa$ be an arbitrary cardinal and consider $2^\kappa$ with the usual topology and let $X \subset 2^\kappa$ be the subspace of all functions of countable support. It is easy to see that $X$ is countably compact and $X$ is well known to have countable tightness (it is even Fr\'echet-Urysohn). The space $X$ is homogeneous because it is a topological group with respect to coordinatewise addition mod 2, yet $|X|=\kappa^\omega$.

A space $X$ is called \emph{power-homogeneous} if there is a cardinal $\kappa$ such that $X^\kappa$ is homogeneous.

\begin{lemma} \label{lemhom}
(Ridderbos, \cite{R}) Let $X$ be a Hausdorff power-homogeneous space. Then $|X| \leq (d(X))^{\pi \chi(X)}$.
\end{lemma}

The following lemma can be proved by modifying slightly the proof of a result of Shapirovskii (see \cite{Ju}, 3.14).

\begin{lemma} \label{lempichi}
Let $X$ be an initially $\kappa$-compact space such that $F(X) \leq \kappa$. Then $\pi \chi(X) \leq \kappa$.
\end{lemma}

We say that a space $X$ has $G_\kappa$ density at most $\kappa$ at the point $x \in X$ if there is a $G_\kappa$ subset $G$ of $X$ containing $x$ such that $d(G) \leq \kappa$.

Lemmas $\ref{lemarhan2}$ and $\ref{lemarhan}$ are essentially due to Arhangel'skii (see \cite{A3} and \cite{J}, 3.12 for the proofs of closely related statements).

\begin{lemma} \label{lemarhan2}
Let $X$ be an initially $\kappa$-compact regular space such that $F(X) \leq \kappa$. Then $t(X) \leq \kappa$.
\end{lemma}

\begin{lemma} \label{lemarhan}
Let $X$ be an initially $\kappa$-compact regular space such that $F(X) \leq \kappa$. Then $X$ has $G_\kappa$-density at most $\kappa$ at some point.
\end{lemma}

\begin{lemma}
(Arhangelskii, van Mill and Ridderbos, \cite{AMR})
Let $X=\prod \{X_i: i \in I\}$. Suppose $X$ is homogeneous and for every $i \in I$, the $G_\kappa$-density of $X_i$ does not exceed $\kappa$ at some point. If for some $j \in I$, we have $\pi \chi(X_j) \leq \kappa$ then the $G_\kappa$-density of $X_j$ does not exceed $\kappa$ at all points of $X_j$.
\end{lemma}

The above lemma, along with Lemma $\ref{lempichi}$ and Lemma $\ref{lemarhan}$ implies the following statement.

\begin{corollary} \label{phcor}
Let $X$ be a power-homogeneous initially $\kappa$-compact regular space such that $F(X) \leq \kappa$. Then the $G_\kappa$ density of $X$ does not exceed $\kappa$ at all points.
\end{corollary}

The following lemma uses an idea of Arhangel'skii from \cite{A4}.

\begin{lemma} \label{lemdensity}
Let $X$ be an initially $\kappa$-compact regular power-homogeneous space such that $F(X) \cdot wL_c(X) \leq \kappa$. Then $d(X) \leq 2^{\kappa}$.
\end{lemma}

\begin{proof}
By Corollary $\ref{phcor}$, we can choose, for every $x \in X$, a $G_\kappa$ set $G_x$ such that $x \in G_x$ and $d(G_x) \leq \kappa$. Note now that the $G_\kappa$-density of $G_x$ does not exceed its $G_\kappa$-weight, which in turn is at most $w(G_x)^\kappa$. By regularity of $X$, we have $w(G_x) \leq 2^{d(G_x)}$. Hence the density of $G_x$ in the $G_\kappa$ topology does not exced $2^\kappa$. Now $\{G_x: x \in X \}$ is a $G_\kappa$-cover of $X$ and hence by Theorem $\ref{wLtheorem}$ we can find $C \in [X]^{2^\kappa}$ such that $\bigcup \{G_x: x \in C\}$ is dense in $X_\kappa$. For every $x \in C$, fix a set $D_x$, dense in $G_x$ (in the $G_\kappa$ topology). Then $D=\bigcup \{D_x: x \in C\}$ is a dense subset of $X_\kappa$ of cardinality $2^\kappa$. Since the topology of $X_\kappa$ is finer than the topology of $X$, we have that $D$ is also dense in $X$ and hence $d(X) \leq 2^\kappa$. 
\end{proof}

\begin{theorem} \label{mainhomthm}
Let $X$ be an initially $\kappa$-compact power-homogeneous regular space such that $F(X) \cdot wL_c(X) \leq \kappa$. Then $|X| \leq 2^{\kappa}$.
\end{theorem}

\begin{proof}
By Lemma $\ref{lempichi}$ we have $\pi \chi(X) \leq \kappa$ and using Lemma $\ref{lemdensity}$ we obtain that $d(X) \leq 2^\kappa$. Hence, using Lemma $\ref{lemhom}$, we obtain $|X| \leq d(X)^{\pi \chi(X)} \leq 2^\kappa$.
\end{proof}

\begin{corollary}
(De la Vega) Let $X$ be a compact homogeneous space. Then $|X| \leq 2^{t(X)}$.
\end{corollary}

\begin{corollary}
(Arhangel'skii, van Mill and Ridderbos) Let $X$ be a compact power-homogeneous space. Then $|X| \leq 2^{t(X)}$.
\end{corollary}

Note that, while $F(X)=t(X)$ for every compact space $X$, the cardinal invariants $F(X)$ and $t(X)$ are not related for initially $\kappa$-compact spaces, as the following pair of examples shows. The first example exploits an idea from \cite{Ok}.

\begin{example}
For every cardinal $\kappa>\omega_1$, there is a countably compact space $X$ such that $F(X) \leq \omega_1$ and $t(X)=\kappa$.
\end{example}

\begin{proof}
Let $Y=\{x \in 2^\kappa: |x^{-1}(1)| \leq \aleph_0\}$ and $X=Y \cup \{\mathbf{1}\}$, with the topology inherited from $2^\kappa$, where $\mathbf{1}$ indicates the function which is identically equal to $1$. Then $X$ is countably compact and it is easy to see that $t(X)=\kappa$.

Since $Y$ is Fr\'echet-Urysohn, we have $t(Y) = \aleph_0$. It is not too hard to see that $L(Y)=\aleph_1$ (see, for example, the proof of Theorem 3.7 from \cite{SWhyburn}). If $F$ is a free sequence in $X$, then $F \cap Y$ is a free sequence in $Y$ having the same cardinality. Therefore $F(X) \leq F(Y) \leq t(Y) \cdot L(Y) \leq \aleph_1$.
\end{proof}

\begin{example}
For every cardinal $\kappa$ of uncountable cofinality, there is a countably compact space such that $t(X)=\omega$ and $F(X)=\kappa$.
\end{example}

\begin{proof}
Let $X=\{\alpha < \kappa: cf(\alpha) \leq \omega \}$. Then $X$ is first-countable and hence it has countable tightness. It is also easily seen to be countably compact.

Let $F=\{x_\alpha: \alpha < \kappa \}$ be an increasing enumeration of $Succ(\kappa)$. Then, for every $\beta < \kappa$ we have $\overline{\{x_\alpha: \alpha < \beta\}} \subset [0, x_\beta)$ and $\overline{\{x_\alpha: \alpha \geq \beta \}} \subset [x_\beta, \kappa)$. Hence $F$ is a free sequence of cardinality $\kappa$.

\end{proof}

\section{Acknowledgements}

The first-named author is grateful to FAPESP for financial support through postdoctoral grant 2013/14640-1, \emph{Discrete sets and cardinal invariants in set-theoretic topology} and to Ofelia Alas for useful discussion. The second-named author acknowledges support from NSERC grant 238944.

\end{document}